\documentclass[12pt]{article}
\usepackage{amsfonts}
\usepackage{amsmath}
\usepackage{amssymb}
\usepackage{theorem}

\newtheorem{theorem}{Theorem}
\newtheorem{proposition}[theorem]{Proposition}
\newtheorem{lemma}[theorem]{Lemma} 
\newtheorem{corollary}[theorem]{Corollary}

\theorembodyfont{\rmfamily}
\newtheorem{example}[theorem]{Example}

\newtheorem{remark}[theorem]{Remark}
\newcommand{\cd}{{\mathrm{cd}}}
\newcommand{\tcd}{{\mathrm{tcd}}}
\newcommand{\qq}{{\mathbb Q}}
\newcommand{\rr}{{\mathbb R}}
\newcommand{\zz}{{\mathbb Z}}
\newcommand{\ttt}{{\mathbb T}}
\newcommand{\tilt}{{\widetilde T}}
\newcommand{\tilmu}{{\tilde\mu}}
\newcommand{\redH}{{\overline H}}
\newcommand{\Hom}{{\mathrm{Hom}}}
\newcommand{\Ext}{{\mathrm{Ext}}}
\newcommand{\coind}{{\mathrm{Coind}}}

\newcommand\mapright[1]{\smash{\mathop{\longrightarrow}\limits^{#1}}}
\newcommand\mapdown[1]{\Big\downarrow\rlap{$\vcenter{\hbox{$\scriptstyle#1$}}$}}

\title{The cohomology of Bestvina-Brady groups}
\author{Ian J. Leary\thanks{Partially supported by NSF grant
DMS-0505471}
\and M\"uge Saadeto\u{g}lu\thanks{Supported by the British Council and
by the Ohio State Mathematical Research Institute}}

\date{\today}

\newenvironment{proof}[1][]{\begin{trivlist} \item[\hskip\labelsep
\emph{Proof#1.}]}{\foorp \end{trivlist}}%    Proof
\newcommand{\foorp}{{\unskip\nobreak\hfil\penalty50
 \hskip1em\vadjust{}\nobreak\hfil \vrule height3pt width3pt depth0pt
 \parfillskip=0pt \finalhyphendemerits=0 \par}}
\begin{document} 

\maketitle

\emph{Dedicated to Warren Dicks, on the occasion of his 60th birthday. }

\begin{abstract} 
For each subcomplex of the standard CW-structure on any torus, 
we compute the homology of a certain infinite cyclic regular
covering space.  In all cases when the homology is finitely generated, 
we also compute the cohomology ring.  For aspherical subcomplexes of 
the torus, our computation gives the homology of the groups introduced 
by M. Bestvina and N. Brady in \cite{bb}.  We compute the
cohomological dimension of each of these groups.  
\end{abstract} 

\section{Introduction} 
Let $\ttt$ be the circle, or 1-dimensional unitary group, given a 
CW-structure with one 0-cell and one 1-cell.  Suppose also that the 
identity element of the group is chosen to be the 0-cell.  
For a set $V$, let $T(V)$ denote the direct sum $T(V)= \bigoplus_{v\in
V}\ttt$.  There is a natural 
CW-structure on $T(V)$ in which the $i$-cells are in bijective
correspondence with $i$-element subsets of $V$.  

For the purposes of this paper, a simplicial complex will be defined 
abstractly as a non-empty set of finite sets which is closed
under inclusion.  The one element members of the set of sets are the 
vertices of the simplicial complex.  Every simplicial complex 
(including the empty simplicial complex) contains 
a unique $-1$-simplex corresponding to the empty set.  

If $\sigma$ is an finite subset of $V$, the closure in $T(V)$ of 
the cell corresponding to $\sigma$ is equal to $T(\sigma)$, and 
consists of all the cells corresponding to subsets of $\sigma$.  It
follows that 
there is a bijective correspondence between simplicial complexes whose 
vertex set is contained in $V$ and non-empty subcomplexes of $T(V)$ 
(see \cite[3.23]{hatcher} for this statement in the 
case when $V$ is finite).  The empty 
simplicial complex corresponds to the subcomplex $T_\emptyset$
consisting of just the single 0-cell of $T(V)$, and a non-empty
simplicial complex $L$ corresponds to the complex $T_L$ defined by 
$$T_L = \bigcup_{\sigma\in L} T(\sigma).$$

The fundamental group of $T_L$ and the cohomology ring of $T_L$ are 
easily described in terms of $L$, and there is a 
characterisation of those $L$ for which $T_L$ is aspherical.  (We
shall describe all of these results below.)  

A point in $T(V)$ is a vector $(t_v)$ of elements of $\ttt$ indexed by
$V$, such that only finitely many $t_v$ are not the identity element.  
The group multiplication induces a map $\mu:T(V)\rightarrow \ttt$ 
which takes the point $(t_v)$ to the product of all of the
non-identity $t_v$'s.  For each $L$ this induces a cellular map 
$\mu_L:T_L\rightarrow \ttt$, and $\mu_L$ is surjective when $L$ is 
non-empty.  Our aim is to 
study the homology and cohomology of the space $\tilt_L$, the 
infinite cyclic cover of $T_L$ obtained by pulling back the universal 
cover of $\ttt$ via $\mu_L$.   
$$\begin{array}{ccc} 
\tilt_L&\mapright{}&\rr \\
\mapdown{} &&\mapdown{} \\
T_L&\mapright{\mu_L}&\ttt\\
\end{array}
$$

For each $L$, we describe the homology of $\tilt_L$, together with 
information about the 
$\zz$-action induced by the action of $\zz$ by deck transformations 
on $\tilt_L$.  We deduce that for any non-trivial 
ring $R$, the 
$R$-homology of $\tilt_L$ is finitely generated as an $R$-module 
if and only if $L$ is finite and $R$-acyclic.  In all cases when 
$L$ is $R$-acyclic, we give a complete description of the cohomology
ring $H^*(\tilt_L;R)$.  

Let $G_L$ denote the fundamental group of $T_L$.  Since the
fundamental group of a CW-complex depends only on its 2-skeleton, the 
group $G_L$ depends only on the 1-skeleton of the simplicial complex
$L$.  The presentation for $G_L$ coming from the cell structure on 
$T_L$ has one generator for each vertex of $L$, subject only to the 
relation that the generators $v$ and $w$ commute whenever $\{v,w\}$ is
an edge in $L$.  These groups are known as right-angled Artin groups. 
It can be shown that $T_L$ is aspherical if and only if $L$ is a flag 
complex.  Every simplicial complex may be completed to a flag complex 
with the same 1-skeleton (just add in a simplex for each finite complete
subgraph of the 1-skeleton) and so one sees that the spaces of the form 
$T_L$ include models for the classifying spaces of all right-angled 
Artin groups.  

When $L$ is non-empty, $\tilt_L$ is connected and the 
fundamental group of $\tilt_L$ is the kernel of the induced map
$\mu_*: G_L\rightarrow \zz$, which sends each of the generators for 
$G_L$ to $1\in \zz$.  Call this group $H_L$.  The groups $H_L$ are 
known as Bestvina-Brady groups.  In the case when $L$ is a finite 
flag complex, M. Bestvina and N. Brady showed that the homological 
finiteness properties of $H_L$ are determined by the homotopy type 
of~$L$.  For example, they show that $H_L$ is finitely presented if
and only if $L$ is 1-connected \cite{bb}.  For an explicit
presentation for $H_L$ for any $L$, see \cite{dl}.  

Let $L$ be an $n$-dimensional flag complex.  It is easy to show 
that in this case, the cohomological dimension of the group $G_L$ 
is equal to $n+1$.  The cohomological dimension of $G_L$ over any 
non-trivial ring $R$ is also equal to $n+1$.  
It also follows easily that the cohomological 
dimension of $H_L$ is equal to either $n$ or $n+1$.  Our computations 
together with some of the results from \cite{bb} allow us to determine 
the cohomological dimension of $H_L$, at least in the case when $R$ is 
either a field or a subring of the rationals.  If $n=0$ and 
$L$ is a single point, 
then $H_L$ is the trivial group.  Otherwise, if there exists an $R$-module 
$A$ such that $H^n(L;A)\neq 0$, then $H_L$ has cohomological dimension 
$n+1$ over $R$.  If there exists no such $A$, then $H_L$ 
has cohomological dimension $n$
over $R$.  Note that in contrast to the case of $G_L$, the
cohomological dimension of $H_L$ may vary with the choice of ring $R$.  
As a corollary we deduce that the trivial cohomological 
dimension and cohomological dimension of $H_L$ are equal.  

Some of these results appeared, with slightly different proofs, in the 
Southampton PhD thesis of the second named author.  For some results, 
we give a brief sketch of a second proof.  Some computations of 
low-dimensional ordinary cohomology (and many other algebraic
invariants) for a special class of the Bestvina-Brady groups 
also appear in a recent preprint of S. Papadima and A. Suciu~\cite{pasu}.

\section{Homology and cohomology of $T_L$} 

The differential in the cellular chain complex for $T(V)$ is trivial,
and hence so is the differential in the cellular chain complex for 
$T_L$, for any $L$.  It follows that for any ring $R$, $H_i(T_L;R)$ is a 
free $R$-module with basis the $i$-cells of $T_L$, or equivalently the 
$(i-1)$-simplices of $L$.  
The differential in the cellular cochain complex is also trivial. 
The group $H^i(T_L;R)$ is isomorphic to a direct product of copies of 
$R$ indexed by the $(i-1)$-simplices of $L$.  To describe the ring 
structure on the cohomology, we first consider the case of the torus 
$T(V)$.  

The cohomology ring 
$H^*(T(V);R)$ can be described as the exterior algebra
$\Lambda^*_{R,V}$.  A homogeneous element $f\in \Lambda^n_{R,V}$ is an 
alternating function $f:V^n\rightarrow R$, where we say that a 
function is alternating if the following two conditions are satisfied:  
\begin{enumerate} 
\item $f(v_1,\ldots,v_n)= 0$ whenever there 
exists $1\leq i<j\leq n$ with $v_i=v_j$; 
\item $f(v_1,\ldots, v_i,v_{i+1},\ldots,v_n) = -f(v_1,\ldots,
  v_{i+1},v_i,\ldots,v_n)$ for any $i$ with $1\leq i <n$.  
\end{enumerate} 
If $f\in \Lambda^i$ and $g\in \Lambda^{n-i}$, the product 
$f.g$ is the so-called `shuffle product'.  This is defined in 
terms of the pointwise product by the equation 
$$f.g(v_1,\ldots,v_n)=\sum_\pi \epsilon(\pi) 
f(v_{\pi(1)},\ldots,v_{\pi(i)})g(v_{\pi(i+1)},\ldots,v_{\pi(n)}),$$ 
where $\epsilon(\pi)\in\{\pm 1\}$ denotes the sign of the permutation $\pi$, 
and the summation ranges over all permutations $\pi$ such that 
$$\pi(1)<\pi(2)<\cdots <\pi(i)\quad\hbox{and}\quad \pi(i+1)<\pi(i+2)< 
\cdots < \pi(n).$$ 
(The `shuffles' or permutations of the above type are chosen because
they are a set of coset representatives in $S_n$ for the subgroup 
$S_i\times S_{n-i}$, so that each $i$-element subset of
$\{1,\ldots,n\}$ is equal to $\{\pi(1),\ldots,\pi(i)\}$ for exactly
one such~$\pi$.  Any other set of coset representatives could be used 
instead.)   

There is a similar description of the ring structure on $H^*(T_L;R)$ 
for any simplicial complex $L$, as the exterior face ring
$\Lambda^*_R(L)$ of $L$.  If $V$ is the vertex set of $L$, 
$\Lambda^*_R(L)$ is the quotient of $\Lambda^*_{R,V}$ by 
the homogeneous ideal $I_L$, with generators the functions 
that vanish on every
$n$-tuple $(v_1,\ldots, v_n)$ which does not span a simplex of $L$.  
The inclusion of $T_L$ in $T(V)$ induces a homomorphism of cohomology
rings 
$$\Lambda^*_{R,V}\cong H^*(T(V);R)\rightarrow H^*(T_L;R),$$
and it is easy to check (from the additive description of $H^*(T_L;R)$ 
given above) that this homomorphism is surjective and that its kernel 
is $I_L$.  Hence one 
obtains a theorem which was first stated in \cite{kimroush} in the
case when $L$ is finite:  

\begin{theorem} \label{facering} 
For any simplicial complex $L$ and ring $R$, the 
cohomology ring $H^*(T_L;R)$ is isomorphic to the exterior face 
ring $\Lambda^*_R(L)$. 
\end{theorem}  

For any path-connected space $X$, there is a natural isomorphism 
between $H^1(X;\zz)$ and $\Hom(\pi_1(X),\zz)$.  The element of 
$H^1(T_L;\zz)=\Lambda^1_\zz(L)$ that corresponds to the 
homomorphism $\mu_*:G_L\rightarrow \zz$ is the element $\beta_L$, 
the constant function which takes each vertex of $L$ to $1\in \zz$.  
By a slight abuse of notation, write $\beta_L$ also for the element 
of $\Lambda_R^1(L)$ that takes each vertex of $L$ to $1\in R$.  

In any anticommutative ring, multiplication by an element of odd 
degree gives rise to a differential.  The cochain complex structure 
on $\Lambda_R^*(L)$ given by multiplication by $\beta_L$ is easily 
described.  

\begin{theorem} \label{bbdiff} 
For any ring $R$, there is a natural isomorphism of
cochain complexes 
$$(\Lambda_R^*(L),\beta\times)\cong C^{*-1}_+(L;R)$$ 
between the exterior face ring of $L$ with differential given by 
left multiplication by $\beta_L$, and the augmented simplicial cochain
complex of $L$ shifted in degree by one.  
\end{theorem} 

\begin{proof} In degree $i$, each of the two graded $R$-modules is 
isomorphic to a direct product of copies of $R$ indexed by the
$(i-1)$-simplices of $L$, or equivalently the $R$-valued functions 
on the oriented $(i-1)$-simplices of $L$, where
$f(-\sigma)=-f(\sigma)$ if $-\sigma$ is the same simplex as $\sigma$
with the opposite orientation.  It remains to show that this
isomorphism is compatible with the differentials on the two cochain
complexes.  

Let $f$ be an $R$-valued function on the $(i-1)$-simplices of $L$,
and compare the functions $\beta.f$ and $\delta f$, the image of 
$f$ under the differential on $C_+^{*-1}(L;R)$.      
If $(v_0,\ldots,v_i)$ is 
the vertex set of an oriented $i$-simplex of $L$, then 
\begin{eqnarray*}
\beta.f(v_0,\ldots, v_i)&=& \sum_{j=0}^i (-1)^j
\beta(v_j)f(v_0,\ldots,v_{j-1},v_{j+1},\ldots,v_i)\\ 
&=&\sum_{j=0}^i (-1)^j f(v_0,\ldots,v_{j-1},v_{j+1},\ldots,v_i) \\
&=&\delta f(v_0,\ldots,v_i).\\
\end{eqnarray*} 
This 
%shows that the above isomorphism of cochain groups is an
%isomorphism of cochain complexes, and 
completes the proof.  
\end{proof} 

\section{Higher homotopy of $T_L$} 

Recall that a full subcomplex $M$ of a simplicial complex $L$ is a 
subcomplex such that if $\sigma$ is any simplex of $L$ and each 
vertex of $\sigma$ is in $M$, then $\sigma$ is in~$M$.  

\begin{proposition} 
If $M$ is a full subcomplex of $L$, then $T_M$ is a retract of $T_L$.  
\end{proposition} 

\begin{proof} 
Let $W$ be a subset of $V$.  There is an isomorphism of topological 
groups $T(V)\cong T(W)\oplus T(V-W)$.  The inclusion 
$$i: T(W)\cong T(W)\oplus \{1\} \rightarrow T(V)$$ 
and projection 
$$\pi: T(V) \rightarrow T(V)/\{1\}\oplus T(V-W)\cong T(W)$$ 
satisfy $\pi\circ i=1_{T(W)}$, and show that $T(W)$ is a retract 
of $T(V)$.  

Now suppose that $L$ is a simplicial complex with vertex set $V$ and 
that $M$ is the full subcomplex with vertex set $W\subseteq V$.  Then 
$T_L$ is a subcomplex of $T(V)$, and $T_M$ is a subcomplex of $T(W)$.  The 
maps $i$ and $\pi$, when restricted to $T_M$ and $T_L$, show that
$T_M$ is a retract of $T_L$ as claimed.  
\end{proof} 

%Let $W$ be a subset of $V$.  The subspace of the torus $T(V)$
%consisting of vectors $(t_v)$ of elements of $\ttt$ such that 
%$t_v=1$ whenever $v\notin W$ can be identified with the torus 
%$T(W)$.  Furthermore, the map from $T(V)$ to this subspace 
%given by replacing $t_v$ by $1$ for each $v\notin W$ is a 
%post-inverse to the inclusion of $T(W)$, and hence $T(W)$ is a
%retract of $T(V)$.  If $L$ is any simplicial complex with vertex
%set $V$, and $M$ is the full subcomplex of $L$ with vertex set 
%$W$, it follows that $T_M$ is a retract of $T_L$.  

Recall that a simplicial complex $L$ is said to be flag if every 
finite complete subgraph of the 1-skeleton of $L$ is the 1-skeleton 
of a simplex of $L$.  Any full subcomplex of a flag complex is flag.  

\begin{proposition} \label{aspher} 
$T_L$ is aspherical if and only if $L$ is a 
flag complex.  
\end{proposition} 

\begin{proof} A subset of a CW-complex that meets the interior of 
infinitely many cells contains an infinite discrete set, so cannot 
be compact.  Hence any map from a sphere to a CW-complex has image 
inside a finite subcomplex and any homotopy between maps of a sphere
into a CW-complex has image contained in a finite subcomplex.  Thus 
it suffices to consider the case when $L$ is finite.  

Suppose that $L$ is a finite flag complex with vertex set $V$.  
If $L$ is an $n$-simplex, then $T_L$ is an $(n+1)$-torus, and 
so $T_L$ is aspherical.  If $L$ is not a simplex, then there exist 
$v_1, v_2\in V$ so that there is no edge in $L$ 
from $v_1$ to $v_2$.  For $i=1,2$, let $L_i$ be the full subcomplex of 
$L$ with vertex set $V-\{v_i\}$, and define $L_3$ by $L_3=L_1\cap
L_2$.  Then $L=L_1\cup L_2$, and each of 
$L_1$, $L_2$ and $L_3$ is flag.  By induction, $T_{L_i}$
is aspherical for $i=1,2,3$.  Also $T_{L_3}= T_{L_1}\cap T_{L_2}$ is a
subcomplex of both $T_{L_1}$ and $T_{L_2}$.  The fundamental group of 
$T_{L_3}$ maps injectively to the fundamental group of each of
$T_{L_1}$ and $T_{L_2}$, since $T_{L_3}$ is a retract of each of 
$T_{L_1}$ and $T_{L_2}$.  A theorem of Whitehead~\cite[1.B.11]{hatcher}
implies that $T_L= T_{L_1}\cup T_{L_2}$ is aspherical.  

Conversely, suppose that $L$ is not flag.  Then $L$ contains a
full subcomplex $M$ which is equal to the boundary of an $n$-simplex for 
some $n>1$.  Then $T_M$ is the $n$-skeleton of an $(n+1)$-torus, 
and so $\pi_n(T_M)$ is non-zero.  Since $T_M$ is a retract of 
$T_L$ it follows that $\pi_n(T_L)$ is non-zero too.  
\end{proof} 

\begin{remark} There is also a metric proof that $T_L$ is aspherical 
whenever $L$ is a finite flag complex: in this case the geodesic 
metric induced by the standard product metric on $T(V)$ is locally 
CAT(0).  A version of the Cartan-Hadamard theorem shows that any 
locally CAT(0) metric space is aspherical \cite{bb}.  
We give the above proof 
instead to emphasize that the metric technology is not needed.  
\end{remark} 

\section{Homology of $\tilt_L$} 
\label{sec:hom}

Let $Z$ denote the fundamental group of $\ttt$, an infinite cyclic 
group.  Since $\tilt_L$ is defined in terms of $\mu_L:T_L\rightarrow 
\ttt$ by pulling back the universal covering space of $\ttt$, $Z$ acts
via deck transformations on $\tilt_L$.  When $L$ is non-empty, the 
map $\mu_L$ induces an isomorphism $G_L/H_L\cong Z$.  In this section 
we describe the cellular chain complex and homology of $\tilt_L$ as a 
$Z$-module, for every $L$.  Let $C^+_*(L)$ denote the augmented 
cellular chain complex of $L$, and let $d=d_L$ be its differential.  

\begin{proposition} \label{chaincomp}
The cellular chain complex $C_*(\tilt_L)$ is 
isomorphic to $\zz[Z]\otimes C^+_{*-1}(L)$ with differential $(1-z)\otimes 
d_L$.  
\end{proposition} 

\begin{proof} Each $(n-1)$-simplex $\sigma$ of $L$ corresponds to a cubical 
$n$-cell in $T_L$, whose opposite faces are identified.  In $\tilt_L$ this 
lifts to a free $Z$-orbit of $n$-cells.  The $i$th opposite pair of
faces are no longer identified, but differ by the translation action 
of $x$.  By picking an orbit representative in each orbit of cells, 
we establish a $Z$-equivariant bijection between the set of $n$-cells 
of $\tilt_L$ and the direct product of $Z$ with the set of
$(n-1)$-simplices of $L$.  The free abelian group with this basis is 
naturally isomorphic to $\zz[Z]\otimes C^+_{n-1}(L)$.  
Let $v$ be any fixed 0-cell of $\tilt_L$.  In each orbit of 
higher-dimensional cells, pick the orbit representative that has $v$ 
as a vertex but does not have $z^{-1}v$ as a vertex.  With respect 
to this choice of orbit representatives, the boundary map is as 
claimed.  
\end{proof} 

\begin{corollary}\label{homology} 
For any $L$, for any abelian group $A$, and for any $n\geq 0$, 
%commutative ring $R$, and for any $n\geq 0$,
there are short exact sequences of $\zz [Z]$-modules: 
$$0\rightarrow B^+_{n-1}(L;A) \rightarrow 
H_n(\tilt_L;A)\rightarrow \zz[Z]\otimes \redH_{n-1}(L;A)\rightarrow
0,$$ 
$$0\rightarrow \zz[Z]\otimes \redH_{n-1}(L;A)\rightarrow
H_n(\tilt_L;A)\rightarrow Z^+_{n-1}(L;A) \rightarrow 
0,$$ 
where $Z$ acts trivially on $Z^+_*(L;A)$ and on $B^+_*(L;A)$, the 
cycles and boundaries in $C^+_*(L;A)$.    
The inclusion of the $Z$-fixed points in $H_n(\tilt_L;A)$ gives rise
to the first sequence, and the map of $H_n(\tilt_L;A)$ onto its 
largest $Z$-invariant quotient gives rise to the second sequence.  

In the case when $A=R$, a ring, each sequence admits an 
$R[Z]$-module structure.  In this case, the first sequence is split if  
$\redH_{n-1}(L;R)$ is $R$-projective, and the second sequence always
admits an $R$-module splitting.  
\end{corollary} 

\begin{proof} 
Take elements $p(z)\in \zz[Z]$, and $c\in C^+_{n-1}(L;A)$.  The chain
$p(z)\otimes c$ is a cycle for $(1-z)\otimes d$ if and only if 
$d(c)=0$.  The boundary of $p(z)\otimes c$ is $(1-z)p(z)\otimes dc$.  
Thus the cycles $Z_n$ in $C_*(\tilt_L;A)$ may be identified with 
$\zz[Z]\otimes Z^+_{n-1}(L;A)$, and the boundaries $B_n$ may be 
identified with $(1-z)\zz[Z]\otimes B^+_{n-1}(L;A)$.  Between these 
lies $B_n'=\zz[Z]\otimes B^+_{n-1}(L;A)$, and $B'_n$ is a
$\zz[Z]$-submodule of $Z_n$.  This gives a short exact sequence of 
$\zz[Z]$-modules 
$$0\rightarrow B'_n/B_n \rightarrow 
H_n(\tilt_L;A)\rightarrow Z_n/B'_n \rightarrow 0,$$ 
and one sees that $B'_n/B_n\cong B^+_{n-1}(L;A)$ with the trivial 
$Z$-action and that $Z_n/B'_n\cong \zz[Z]\otimes \redH_{n-1}(L;A)$.  

Now define $Z'_n$ to be $(1-z)\zz[Z]\otimes Z^+_{n-1}(L;A)$, a
$\zz[Z]$-submodule of $Z_n$.  As abelian groups, $Z_n = Z'_n \oplus
\left(1\otimes Z^+_{n-1}(L;A)\right)$.  
It follows that the short exact sequence 
$$0\rightarrow Z'_n/B_n \rightarrow 
H_n(\tilt_L;A)\rightarrow Z_n/Z'_n \rightarrow 0$$ 
is always $\zz$-split.  One sees that $Z'_n/B_n\cong \zz[Z]\otimes
\redH_{n-1}(L;A)$ and that $Z_n/Z'_n\cong Z^+_{n-1}(L;A)$.  

To compute the $Z$-fixed points in $H_n(\tilt;A)$, apply the 
$Z$-fixed point functor to the first sequence.  Since this 
functor is left-exact, one obtains an exact sequence:  
$$0\rightarrow B^+_{n-1}(L;A) \rightarrow 
H_n(\tilt_L;A)^Z\rightarrow 0. $$ 

To compute the maximal $Z$-fixed quotient of $H_n(\tilt;A)$, start
with the short exact sequence 
$$0\rightarrow B_n \rightarrow Z_n \rightarrow H_n\rightarrow 0,$$ 
and apply the invariant quotient functor $H_0(Z; -)$.  This functor 
is right-exact, and it is easy to see that $H_0(Z;Z_n)= Z_n/Z'_n$.  
Hence one obtains an exact sequence: 
$$Z_n/Z'_n \rightarrow H_n(\tilt_L;A)_Z\rightarrow 0.$$ 
Since we have already shown that $Z_n/Z'_n$ is a $Z$-invariant
quotient of the homology group 
$H_n(\tilt_L;A)$, we see that the maximal $Z$-invariant
quotient is isomorphic to $Z_n/Z'_n$ as claimed.  

In the case when $A=R$, a ring, the $\zz[Z]$-modules 
and maps that appear in the short exact sequences also admit an 
$R$-module structure which commutes with the $Z$-action.  
If $\redH_{n-1}(\tilt;R)$ is $R$-projective, then $\zz[Z]\otimes 
\redH_{n-1}(\tilt;R)$ is $R[Z]$-projective, and so the first 
short exact sequence of $R[Z]$-modules splits.  In any case, 
$Z^+_{n-1}(L;R)$ is free as an $R$-module, and so the second 
short exact sequence admits an $R$-splitting.  
\end{proof} 

%\begin{remark} 

\begin{corollary}\label{acyclichomology} 
Suppose that $L$ is a finite complex and $R$ is a ring such that 
$\redH_i(L;R)=0$ for $i<n$.  Suppose also that $L$ has $f_i$ 
$i$-dimensional simplices for $i\geq 0$, and define $f_{-1}=1$.  For 
each $i$ with $0\leq i\leq n$, $H_i(\tilt_L;R)$ is a free $R$-module 
of rank $$\sum_{j=0}^{i} (-1)^{i+j}f_{j-1}.$$ 
\end{corollary} 

\begin{proof} 
For each $j$, $C^+_j(L;R)$ is a free $R$-module of rank $f_j$. 
For each $i\leq n$, we know that 
$$H_i(\tilt_L;R)\cong Z^+_{i-1}(L;R) = B^+_{i-1}(L;R),$$
$$C^+_i(L;R)\cong Z^+_{i-1}(L;R)\oplus Z^+_i(L;R),$$  
and that $Z^+_i(L;R)$ is a free $R$-module.  Solving for the 
rank of $Z^+_i(L;R)$ gives the claimed result.  
\end{proof} 

%DO I HAVE TO SAY SOMETHING ABOUT IBN?  
%No, commutative rings have fields as quotients.  

\begin{corollary}
\label{eulerchar} 
If $L$ is finite and $\qq$-acyclic, then the Euler characteristic of 
$\tilt_L$ is defined and is given by the formula:  
$$\chi(\tilt_L)=\sum_{i\geq 0} (-1)^i(i+1)f_i.$$ 
\end{corollary} 

\begin{proof} 
Since $L$ is $\qq$-acyclic, the reduced Euler characteristic 
$\sum_{i\geq -1} (-1)^if_i$ of $L$ is equal to zero.  
Let $L$ be $n$-dimensional.  Then the expression for the rank 
of $H_i(\tilt_L;\qq)$ given by the previous corollary gives 
$$\chi(\tilt_L)= \sum_{i=0}^n (-1)^i\sum_{j=0}^i (-1)^{i+j}f_{j-1} 
%=\sum_{0\leq j\leq i \leq n} (-1)^j f_{j-1} 
= \sum_{i=0}^n (-1)^i(n+1-i) f_{i-1}.$$ 
Hence we see that 
$$\chi(\tilt_L)
= \sum_{i=0}^n (-1)^i(n+1-i) f_{i-1} - (n+1)\sum_{i=0}^{n+1} (-1)^i
f_{i-1} = \sum_{i=0}^{n+1}(-1)^{i-1}if_{i-1},$$ 
as claimed.  
\end{proof} 

\begin{remark} There is a way to deduce Corollary~\ref{eulerchar} 
from results of Bestvina and Brady from~\cite{bb}.  This 
alternative proof was given in the second named 
author's PhD thesis \cite{muge}.  Let $\tilmu_L$ 
be the map defined by the pullback square 
$$\begin{array}{ccc} 
\tilt_L&\mapright{\tilmu_L}&\rr \\
\mapdown{} &&\mapdown{} \\
T_L&\mapright{\mu_L}&\ttt.\\
\end{array}
$$
In the case when $L$ is finite and $\qq$-acyclic, Bestvina and Brady 
show that the inclusion $\tilmu_L^{-1}(x)$ in $\tilt_L$ is a rational 
homology isomorphism for any real number $x$ \cite{bb}.  
In the case when $x$ 
is not an integer, there is a cellular structure on $\tilmu_L^{-1}(x)$ 
with $(i+1)f_i$ i-cells for each $i\geq 0$.    

On the other hand, we know of no other proof of
Corollary~\ref{homology} or Corollary~\ref{acyclichomology} 
than the proofs given above.

\end{remark} 

\section{Cohomology of $\tilt_L$} 
\label{sec:coh}

%*****
%Our main interest is cohomology with coefficients in a commutative
%ring $R$, but for application 

\begin{proposition} \label{cochaincomp} 
Let $C^*(\tilt;A)$ denote the cellular cochain complex for $\tilt_L$
with coefficients in the abelian group $A$.  Each $C^n(\tilt_L;A)$ is
isomorphic to a coinduced $Z$-module:  
$$C^n(\tilt_L;A)\cong 
\Hom(\zz[Z],C^{n-1}_+(L;A))\cong \prod_{i\in \zz} C^{n-1}_+(L;A).$$ 
If $A$ is an $R$-module for some ring $R$, then this is 
an isomorphism of $R[Z]$-modules.  The coboundary map is 
given by $\delta((f_i)_{i\in \zz}) = (\delta f_i - \delta
f_{i+1})_{i\in \zz}$.  The action of $Z$ is the `shift action': 
$z(f_i)_{i\in \zz} = (f_{i+1})$.  The image of $C^*(T_L)$ in 
$C^*(\tilt_L)$ is identified with the `constant sequences', 
i.e., those with $f_i=f_j$ for all $i,j$.  
\end{proposition} 

\begin{proof} 
Most of the assertions follow immediately 
from the description of $C_*(\tilt_L)$ given in 
Proposition~\ref{chaincomp}, since 
\begin{eqnarray*} 
C^*(\tilt_L;A)\cong \Hom(C_*(\tilt_L),A)& \cong &
\Hom(\zz[Z]\otimes C_{*-1}^+(L),A)\\ 
&\cong &
\Hom(\zz[Z],C_+^{*-1}(L;A)).
\end{eqnarray*}
The claim concerning the image of $C^*(T_L)$ is clear, since 
cochains that factor through the projection $\tilt_L\rightarrow 
\tilt_L/Z=T_L$ may be identified with cochains that are fixed by 
the $Z$-action.  
\end{proof} 

\begin{corollary}\label{additivecoh}  
For any $L$, any abelian group $A$, and any $n\geq 0$ there are 
short exact sequences of $\zz[Z]$-modules:   
$$0\rightarrow C_+^{n-1}(L;A)/B_+^{n-1}(L;A) \rightarrow 
H^n(\tilt_L;A)\rightarrow M\rightarrow 0,$$ 
$$0\rightarrow \prod_Z\redH^{n-1}(L;A)\rightarrow
H^n(\tilt_L;A)\rightarrow C_+^{n-1}(L;A)/Z_+^{n-1}(L;A) \rightarrow 
0,$$ 
where $Z$ acts trivially on $C_+^*(L;A)$, and $Z_+^*(L;A)$ 
(resp.~$B_+^*(L;A)$) denotes the cocycles (resp.\ coboundaries) in 
$C_+^*(L;A)$.  The module $M$ fits in to a
short exact sequence: 
$$0\rightarrow \redH^{n-1}(L;A)\mapright{\Delta} 
\prod_Z\redH^{n-1}(L;A)\rightarrow M\rightarrow 0,$$ 
where $Z$ acts by the `shift action' on the product and 
where $\Delta$ is the inclusion of the constant sequences.  
The first short exact sequence is the inclusion of the $Z$-fixed
points in $H^n(\tilt_L;A)$.  If $A$ is an $R$-module, then both 
short exact sequences admit an $R[Z]$-action.  
\end{corollary} 

\begin{proof} 
Using the description of $C^*=C^*(\tilt_L;A)$ given in
Proposition~\ref{cochaincomp}, we obtain descriptions of the 
coboundaries $B^n$ and cocycles $Z^n$ in $C^n$:  
$$ B^n \cong \prod_{i\in \zz} B^{n-1}_+(L;A),$$ 
$$ Z^n \cong \{(f_i)\in \prod_{i\in \zz} C^{n-1}_+(L;A) : 
\delta f_i = \delta f_{i+1} \forall i\in \zz\}. $$ 

Let $B'^n$ be the submodule of $Z^n$ generated by $B^n$ and the
constant sequences $\{(f_i)_{i\in \zz}: f_i=f_j\}$, and let 
$Z'^n$ be the submodule of $Z^n$ consisting of sequences of 
cocycles, i.e., $Z'^n =\prod_{i\in \zz}Z^{n-1}_+(L;A)$.  The first 
short exact sequence in the statement is equal to  
$$
0\rightarrow B'^n/B^n\rightarrow H^n(\tilt_L;A)\rightarrow Z^n/B'^n
\rightarrow 0,$$
and the second one to 
$$0\rightarrow Z'^n/B^n\rightarrow H^n(\tilt_L;A)\rightarrow Z^n/Z'^n
\rightarrow 0.$$
The computation of the $Z$-fixed points in $H^n(\tilt_L;A)$ follows by 
applying the $Z$-fixed point functor to the first exact sequence.  
\end{proof} 

\begin{theorem} \label{cohomology} Let $R$ be a ring.  
The image of the map 
$$H^*(T_L;R)\rightarrow H^*(\tilt_L;R)$$ 
is equal to the $Z$-fixed point subring of $H^*(\tilt_L;R)$ and is 
isomorphic to the quotient $H^*(T_L;R)/(\beta_L)$.  In degree $n$, the
cokernel of this map is isomorphic to an infinite product of copies of 
$\redH^{n-1}(L;R)$.  In particular, the map is a ring isomorphism
if and only if $L$ is $R$-acyclic.  
\end{theorem} 

\begin{proof} 
This follows from Corollary~\ref{additivecoh} and
Theorem~\ref{bbdiff}.  
\end{proof} 

\begin{corollary} \label{acyccohring} 
Suppose that $L$ is $R$-acyclic.  There is an $R$-algebra isomorphism
$$H^*(\tilt_L;R)\cong \Lambda^*_R(L)/(\beta_L).$$ 
For each $n$, $H^n(\tilt_L;R)$ is isomorphic to a direct product of
copies of $R$.  
\end{corollary} 

\begin{proof} 
This follows from Theorem~\ref{cohomology} and Theorem~\ref{facering} 
\end{proof} 

\begin{remark} The second named author's PhD thesis contained a different 
proof of Corollary~\ref{acyccohring} in the case when $L$ is finite, 
flag and $R$-acyclic \cite{muge}.  Recall that we denote by $G_L$ the 
fundamental group of $T_L$, and by $H_L$ the fundamental group of 
$\tilt_L$.  Also recall from Proposition~\ref{aspher} that when $L$ 
is flag, each of $T_L$ and $\tilt_L$ is an Eilenberg-Mac~Lane space 
for its fundamental group.  
In~\cite{muge} an explicit chain homotopy was 
used to show that the action of $Z$ on $H^*(\tilt_L;R)$ is trivial 
when $L$ is finite, flag and $R$-acyclic.  The long exact sequence 
in group cohomology coming from the isomorphism $G_L=H_L:Z$ was then 
used to establish the isomorphism of Corollary~\ref{acyccohring}.  
\end{remark} 

\section{Cohomological dimension} 

%Let $X$ be a path-connected space with fundamental group $G$.  

The trivial cohomological dimension of a space $X$, $\tcd(X)$, is 
defined to be the 
supremum of those integers $n$ for which there exists an abelian 
group $A$ for which the singular cohomology group $H^n(X;A)$ is 
non-zero.  For any non-trivial ring $R$, 
$\tcd_R(X)$ is defined similarly except that only abelian groups 
$A$ admitting an $R$-module structure are considered.  

Now suppose that $X$ is path-connected, and that $G$ is the fundamental 
group of $X$.  
For $M$ a $G$-module, we write $H^*(X;M)$ for the singular cohomology 
of $X$ with twisted coefficients in $M$.  If $X$ admits a universal 
covering space $\widetilde X$, then this is just the cohomology of 
the cochain complex $\Hom_G(C_*(\widetilde X),M)$ of $G$-equivariant 
singular cochains on $\widetilde X$.  The cohomological dimension 
$\cd(X)$ of $X$ is the supremum of those integers $n$ such that there 
is a $G$-module $M$ for which $H^n(X;M)\neq \{0\}$.  For a non-trivial 
ring $R$, $\cd_R(X)$ is defined similarly except that only
$RG$-modules $M$ are considered.  

Each of the invariants depends only on the homotopy type of $X$.  
In the case when $X$ is a classifying space or Eilenberg-Mac~Lane
space for $G$, we write $\cd(G)$, $\cd_R(G)$, $\tcd(G)$ and
$\tcd_R(G)$ for the corresponding invariants of $X$.  

We summarize the properties of these invariants below.  

\begin{proposition} \label{cdprops} 
In the following statements, $X$ is any path-connected space, and $R$
is any non-trivial ring.  Rings are assumed to have units, and ring 
homomorphisms are assumed to be unital.  
\begin{enumerate} 
\item $\tcd(X)=\tcd_\zz(X)$, $\cd(X)=\cd_\zz(X)$.  
\item $\tcd_R(X)\leq \cd_R(X)$.  
\item If there is a ring homomorphism $\phi:R\rightarrow S$, then 
$\tcd_S(X)\leq
\tcd_R(X)$ and $\cd_S(X)\leq \cd_R(X)$. 
\item If $R$ is a subring of $S$ in such a way that $R$ is a direct 
summand of $S$ as an $R$-$R$-bimodule, then $\tcd_R(X)=\tcd_S(X)$ and 
$\cd_R(X)=\cd_S(X)$.  
\item If $Y$ is a covering space of $X$, then $\cd_R(Y)\leq \cd_R(X)$.  
\item If $H\leq G$, then $\cd_R(H)\leq \cd_R(G)$.  
\item For any group $G$, $\cd_R(G)= {\mathrm{proj.dim.}}_{RG}(R)$.  
\end{enumerate} 
\end{proposition}

\begin{proof} A $\zz G$-module is the same thing as a $G$-module,
establishing (1).  Similarly, any $R$-module $A$ may be viewed 
as an $RG$-module by letting each element of $G$ act via the identity, 
which establishes (2).  A ring homomorphism $\phi$ as above allows one 
to define an $R$-module structure on 
any $S$-module and an $RG$-module structure on any $SG$-module,
which proves (3).  Under the hypotheses of (4), any $R$-module $A$ is 
isomorphic to a direct summand of the $S$-module $S\otimes_R A$ and 
any $RG$-module $M$ is isomorphic to a direct summand of the
$SG$-module $SG\otimes_{RG}M$.  Since cohomology commutes with finite
direct sums, this shows that $\tcd_R(X)\leq \tcd_S(X)$ and 
$\cd_R(X)\leq \cd_S(X)$.  The opposite inequalities follow from (3).  

If the fundamental group of $Y$ is the subgroup
$H\leq G$, and $M$ is any $RH$-module, define the coinduced
$RG$-module by 
$$\coind(M) = \Hom_{\zz H}(\zz G,M).$$ 
Then there is an isomorphism $H^*(X;\coind(M))\cong H^*(Y;M)$, which 
establishes (5).  Now (6) is just the special case of (5) in which 
$X$ is a classifying space for $G$, since then the covering space 
of $X$ with fundamental group $H$ is a classifying space for $H$.  

If $X$ is a classifying space for $G$, then the universal covering
space $\widetilde X$ is contractible, and so $C_*(\widetilde X)$ 
is a free resolution of $\zz$ over $\zz G$.  Similarly,
$C_*(\widetilde X;R)$ is a free resolution of $R$ over $RG$.  
Hence $H^*(X;M)$ is isomorphic to $\Ext_{\zz G}(\zz,M)$, and if 
$M$ is an $RG$-module $H^*(X;M)$ is also isomorphic to $\Ext_{RG}(R,M)$, 
which establishes (7).  (See for example \cite{brown} for more details.) 
\end{proof} 

\begin{remark} It is easy to find groups $G$ for which
$\tcd(G)<\cd(G)$.  There are many acyclic groups, or groups 
for which the group homology $H_i(G;\zz)=0$ for all $i>0$
(see for example 
%the discussion of G. Higman's group amongst others in
\cite{bdh}).  For any 
such $G$, $\tcd(G)=0$, while $\cd(G)\neq 0$ unless $G$ is the 
trivial group.  
\end{remark} 

Before stating the next proposition, we remind the reader that 
the dimension of a simplicial complex is the supremum of the 
dimensions of its simplices, so that the empty simplicial 
complex has dimension $-1$.  

\begin{proposition} For any simplicial complex $L$ and any 
non-trivial ring~$R$, 
$$\tcd_R(T_L)=\cd_R(T_L)= \dim(T_L)=\dim(L)+1.$$ 
\end{proposition} 

\begin{proof} 
Immediate from Theorem~\ref{facering}. 
\end{proof} 

\begin{proposition}\label{trivcohdim} 
For any simplicial complex $L$ with at least two vertices and for 
any $R$, 
$$\tcd_R(\tilt_L)= \max\{\dim(L),1+\tcd_R(L)\}.$$ 
\end{proposition} 

\begin{proof} 
Immediate from Corollary~\ref{additivecoh}.  
\end{proof}

%Let $L$ be a simplicial complex.  If $\dim(L)=\infty$, then for 
%any ring $R$, 
%$\tcd_R(\tilt_L)=\infty$.  If $\dim(L)=n<\infty$, then either 
%$\tcd_R(L)<n$ and $\tcd_R(\tilt_L)=n$, or $\tcd(\tilt_L)=n+1$. 
%\end{proposition} 

The following theorem is the first result in this paper for which 
we rely on techniques from Bestvina-Brady~\cite{bb}.  

\begin{theorem} \label{cdnonflag}
%Let $R$ be a non-trivial ring, and suppose that $L$ is $R$-acyclic
%Then $\tcd_R(\tilt_L)=\cd_R(\tilt_L)=\dim(L)$.  
%If $L$ is $R$-acyclic, then 
Let $L$ be a simplicial complex with at least two vertices 
and let~$R$ be a non-trivial ring.  
\begin{enumerate} 
\item
$\dim(L) \leq \tcd_R(\tilt_L)\leq \cd_R(\tilt_L)\leq \dim(L)+1$.  
\item
If $L$ is $R$-acyclic, then $\tcd_R(L)=\cd_R(\tilt_L)= \dim(L)$.  
\end{enumerate} 
\end{theorem} 

\begin{proof} The claims for arbitrary $L$ follow easily from 
earlier results.  Proposition~\ref{trivcohdim} implies that 
$\dim(L)\leq \tcd_R(\tilt_L)$.  By Proposition~\ref{cdprops}, one 
has that $\tcd_R(\tilt_L)\leq \cd_R(\tilt_L)\leq \cd_R(T_L)$, 
and of course $\cd_R(T_L)\leq \dim(T_L)=\dim(L)+1$.

It remains to show that when $L$ is $R$-acyclic, $\cd_R(\tilt_L)\leq
\dim(L)$.  Following Bestvina and Brady~\cite{bb}, 
let $X_L$ be the universal covering space of
  $T_L$, or equivalently of $\tilt_L$.  Now let $f_L:X_L\rightarrow
  \rr$ be the composite
$$X_L\rightarrow X_L/H_L=\tilt_L \mapright{\tilde \mu}\rr,$$ 
where $\tilde\mu$ is the lift of the map $\mu_L:T_L\rightarrow \ttt$.  
Now define $Y_L=f^{-1}(0)\subseteq X_L$.  There is a natural cubical 
CW-structure on $X_L$, whose cells are the lifts to $X_L$ of the cells
of $T_L$.  One can also put a CW-structure on $Y_L$, such that each cell
of $Y_L$ is the intersection of $Y_L$ with a cell of $X_L$.  For this 
CW-structure, the dimension of $Y_L$ is equal to $\dim(L)$.  It can 
be shown that $X_L$ is homotopy equivalent to $Y_L$ with infinitely
many subspaces homotopy equivalent to $L$ coned off (this is from 
\cite{bb}, but see also \cite{vfgroups} which explicitly checks 
this in the case when $L$ is infinite).  

Since $L$ is $R$-acyclic, it follows that the inclusion of $Y_L$ in $X_L$
induces an isomorphism of $R$-homology.  The cellular chain complexes
$C_*(Y_L;R)$ and $C_*(X;R)$ consist of free $RH_L$-modules, and so it 
follows that for any $RH_L$-module $M$, 
$$H^*(\tilt_L;M)\cong H^*(Y_L/H_L;M).$$ 
This shows that $\cd_R(\tilt_L)\leq \dim(L)$ as required.  
\end{proof} 

\begin{remark} In the case when $L$ is empty, $\tilt_L$ is
  0-dimensional, and consists of a single free $Z$-orbit of 
points.  In the case when $L$ is a single point, $\tilt_L$ is 
homeomorphic to $\rr$, with $Z$ acting via the translation 
action of $\zz$.  It is clear that the case when $L$ is empty 
is exceptional for $\tilt_L$.  The reason why the case when $L$ is a 
single point needs to be excluded from Proposition~\ref{trivcohdim} 
and from Theorem~\ref{cdnonflag} is that the formulae for
$H*(\tilt_L;R)$ given in Section~\ref{sec:coh} involve the
\emph{reduced} cohomology of $L$, whereas the definition of
$\tcd(L)$ involves the \emph{unreduced} cohomology of $L$.  
This only makes a difference when $L$ is both 0-dimensional 
and $R$-acyclic, i.e., the case when $L$ is a single point.  
\end{remark} 

\section
{Bestvina-Brady groups} 
%{Cohomological dimension: the aspherical case}

In this section, we give a complete calculation of the cohomological
dimension of Bestvina-Brady groups, or equivalently the cohomological 
dimension of the spaces $\tilt_L$ for $L$ a flag complex.  We impose 
some conditions on the coefficient ring $R$ that were not previously 
required.  The reason why we work only with flag complexes in this
section is that we need to know that when $L\leq K$ is a full
subcomplex of $K$, then $\tilt_L$ is homotopy equivalent to a covering 
space of $\tilt_K$.  

\begin{theorem} \label{cdbbgp} 
Let $R$ be either a field or a subring of $\qq$.  
Let $L$ be a flag complex with at least two vertices.  (This 
implies that $H_L$ is infinite, and that $\tilt_L$ is a classifying 
space for $H_L$.)  The following equations hold:  
$$\cd_R(H_L)=\tcd_R(H_L)=\max\{\dim(L),1+\tcd_R(L)\}.$$ 
\end{theorem} 

The proof of the theorem will require two lemmas, the second of which 
is a strengthening of a lemma from~\cite{bes}.  

\begin{lemma} \label{relbary} 
Suppose that $L$ is a 
flag complex and is a subcomplex of a simplicial complex $K$.  
The relative barycentric subdivision $(K,L)'$ is a flag complex 
containing $L$ as a full subcomplex.  
\end{lemma} 

\begin{proof} Before recalling the definition of the relative 
barycentric subdivision $(K,L)'$, we remind the reader that each 
simplicial complex contains a unique $-1$-simplex corresponding 
to the empty subset of its vertices.  
The vertex set of 
$(K,L)'$ is the disjoint union of the vertex set of $L$ and the 
set of simplices of $K$ not contained in $L$.  An $n$-simplex of 
$(K,L)'$ has the form $(\sigma_0<\sigma_1<\cdots<\sigma_r)$, for 
some $r$ satisfying $0\leq r\leq n+1$.  
Here $\sigma_0$ denotes an $(n-r)$-simplex of $L$, 
each $\sigma_i$ is a simplex of $M$, and $\sigma_i$ is not contained 
in $L$ if $i>0$.  

Suppose that a finite subset $S$ of the vertices of $(K,L)'$ has the 
property that any two of its members are joined by an edge.  Since 
$L$ is flag, this implies that the set $S\cap L$ is the vertex set 
of a simplex $\sigma_0$ of $L$.  Each element of $S-L$ is a simplex 
$\sigma_i$ of $M$ not contained in $L$.  The existence of an edge
between each element of $S\cap L$ and each element of $S-L$ implies 
that each $\sigma_i$ contains $\sigma_0$, and the existence of an 
edge between each pair of elements of $S-L$ implies that the
$\sigma_i$ are totally ordered by inclusion.  This shows that $(K,L)'$ 
is a flag complex.  

From the description of the simplices of $(K,L)'$, it is easy to see 
that any simplex of $(K,L)'$ whose vertex set lies in $L$ is in fact 
a simplex in $L$, which verifies that $L$ is a full subcomplex of
$(K,L)'$.  
\end{proof} 

\begin{lemma} \label{flagemb} 
Let $L$ be a flag complex, let $R$ be either a subring of $\qq$ or a 
field of prime order, and suppose that 
$\tcd_R(L)<\dim(L)$.  Then there is an $R$-acyclic flag complex $K$ 
containing $L$ as a full subcomplex such that $\dim(L)=\dim(K)$.  
\end{lemma} 

\begin{proof} We may assume that $\dim(L)=n$ is finite, and 
we may assume that $n\geq 2$.  Let $C'L$ denote the cone on 
the $(n-2)$-skeleton of $L$, and let $L_1$ be the union $L_1=L\cup
C'L$.  Now $L_1$ is $(n-2)$-connected, and the inclusion map 
$L\rightarrow L_1$ induces an isomorphism $H^n(L_1;A)\rightarrow 
H^n(L;A)$ for any abelian group $A$.  

In each case, $R$ is a principal ideal domain, and so by the
universal coefficient theorem for
cohomology~\cite[V.3.3]{hilsta}, for any $R$-module $A$ we have that 
$$H^n(L_1;A) \cong \Hom_{R}(H_n(L_1,R),A)\oplus
\Ext^1_{R}(H_{n-1}(L_1,R),A).$$
The hypotheses therefore imply that $H_n(L_1;R)=0$, 
and that $H_{n-1}(L_1;R)$ is a projective
$R$-module.  (Since $L_1$ is $(n-2)$-connected, one also 
has that $\redH_i(L_1;R)=0$ for each $i<n-1$.)  
Every projective module for a principal ideal domain is 
free~\cite[I.5.1]{hilsta}, and so $H_{n-1}(L_1;R)$ is a free
$R$-module.  Since $\redH_{n-2}(L_1;\zz)=0$ (because $L_1$ is
$(n-2)$-connected), the universal coefficient theorem for
homology~\cite[V.2.5]{hilsta} tells us that $H_{n-1}(L_1;R)\cong
H_{n-1}(L_1;\zz)\otimes R$.  

Hence there exist integral cycles $z_i\in C_{n-1}(L_1;\zz)$ whose 
images in the group $C_{n-1}(L_1;R)$ map to an $R$-basis for
$H_{n-1}(L_1;R)$.  By the Hurewicz theorem, each $z_i$ is realized
by some map $f_i:S^{n-1}\rightarrow L_1$.  Replace $f_i$ by a
simplicial approximation $f'_i:S_i\rightarrow L_1$, where $S_i$ is 
some triangulation of the $(n-1)$-sphere.  Using a simplicial 
mapping cylinder construction as in~\cite[2C.5]{hatcher}, use 
each $f'_i$ to attach a triangulated $n$-cell to $L_1$ to produce 
$L_2$, an $R$-acyclic simplicial complex with $\dim(L_2)=n$ and 
such that $L\leq L_1\leq L_2$.  
By Lemma~\ref{relbary}, we may take $K$ to be the relative barycentric 
subdivision $K=(L_2,L)'$.  
\end{proof} 

\begin{proof}[ of Theorem~\ref{cdbbgp}] If $S$ is any field, and 
$R$ is the smallest subfield of $S$, then $R$ and $S$ satisfy the 
conditions of statement (4) of Proposition~\ref{cdprops}.  Hence 
there are equalities of functions $\cd_R=\cd_S$ and $\tcd_R=\tcd_S$. 
Thus it suffices to prove Theorem~\ref{cdbbgp} in the case when 
$R$ is either the field of $p$ elements or a subring of $\qq$.  

We may assume that $\dim(L)=n<\infty$.  In the case when
$\tcd_R(L)=n$, Proposition~\ref{trivcohdim} and part (1) of 
Theorem~\ref{cdnonflag} imply that 
$$n+1=\tcd_R(\tilt_L) \leq cd_R(\tilt_L)\leq \dim(\tilt_L)=n+1,$$
as claimed.  
If $\tcd_R(L)<n$, then by Lemma~\ref{flagemb}, there is an
$n$-dimensional $R$-acyclic flag complex $K$ containing $L$ as a 
full subcomplex.  Part (2) of Theorem~\ref{cdnonflag} tells us
that $\tcd_R(K)=\cd_R(K)=n$.  Now $H_L$ is a retract of $H_K$, and so 
by part (6) of Proposition~\ref{cdprops} 
$$\cd_R(H_L)=\cd_R(\tilt_L)\leq\cd_R(H_K)=\cd_R(\tilt_K)=n.$$ 
This gives 
$$n=\tcd_R(\tilt_L)\leq \cd_R(\tilt_L)\leq \cd_R(\tilt_K)=n,$$
as claimed.  
\end{proof} 

\section{Examples} 

\begin{example}\label{cyclicmoore} 
Let $L=L(m)$ be a flag triangulation of the space constructed by attaching
a disc to a circle via a map of degree $m$.  This $L$ has the property 
that $\tcd_R(L)=0$ if $m$ is a unit in $R$ and $\tcd_R(L)=2$ if $m$ is 
not a unit in $R$.  Let $H=H_L$ be the corresponding Bestvina-Brady
group.  From Theorem~\ref{cdnonflag} and
Proposition~\ref{trivcohdim}, it follows that for any ring $R$, 
$\cd_R(H)=2$ if $m$ is a unit in $R$ and $\cd_R(H)=3$ if $m$ is
not a unit in $R$.  
\end{example} 

\begin{example}\label{rationalmoore}
Let $L$ be a flag triangulation of a 2-dimensional Eilenberg-Mac~Lane
space for the additive group of $\qq$.  Then $\tcd_F(L)=1$ for any 
field~$F$, while $\tcd_\zz(L)=2$.  From Theorem~\ref{cdbbgp} it 
follows that the Bestvina-Brady group $H=H_L$ has the property 
that $\cd_F(H)=2$ for any field~$F$, while $\cd_\zz(H)=3$.  
\end{example} 

Groups having similar properties to those given in
Example~\ref{cyclicmoore} have been 
constructed previously by many authors, using finite-index subgroups
of right-angled Coxeter groups~\cite{bes,dile,dran}.  The examples 
coming from Coxeter groups have the advantage that they are of type 
$FP$, whereas they have the disadvantage that the trivial
cohomological dimension of these groups appears to be unknown.  

Finite-index subgroups of (non-finitely generated) 
Coxeter groups were used in~\cite{dile} to 
construct groups having similar properties to those given in 
Example~\ref{rationalmoore}.  Again, the trivial cohomological
dimension of those examples appears to be unknown.  It can be 
shown that for any group $G$ of type $FP$, there is a field $F$ such
that $\cd_F(G)=\cd_\zz(G)$.  (See Proposition~9 of~\cite{dile}.)

%\end{document} 

%\vfill\eject
\leftline{\bf Authors' addresses:}

\noindent 
Ian Leary: Department of Mathematics, The Ohio State University, 
231 West 18th Avenue, Columbus, Ohio 43210-1174, United States.  

and 

\noindent 
School of Mathematics, University of Southampton, Southampton, 
SO17 1BJ, United Kingdom. 

{\tt leary@math.ohio-state.edu} 

\bigskip
\noindent 
M\"uge Saadeto\u{g}lu: School of Mathematics, University of
Southampton, \\
Southampton,  SO17 1BJ, United Kingdom. 

and 

\noindent 
AS277, Department of Mathematics, Eastern Mediterranean University, \\
Gazima\u{g}usa, Mersin 10, Turkey 

{\tt muge.saadetoglu@emu.edu.tr}

\end{document}